\documentclass{amsart}

\usepackage{amsmath}

\usepackage{amsfonts}

\usepackage{hyperref}

\usepackage{latexsym}

\usepackage{amssymb}

\usepackage{amscd}

\newtheorem{theorem}{Theorem}[section]
\newtheorem{lemma}{Lemma}[section]
\newtheorem{corollary}{Corollary}[section]
\newtheorem{proposition}{Proposition}[section]
\newtheorem{definition}{Definition}[section]

\newtheorem{remark}{Remark}[section]

\newcommand{\ma}[1]{\ensuremath{\mathbb{#1}}}
 
\font\bb=msbm7 at 10 pt

\def \C {\hbox{\bb C}}
\def \Z {\hbox{\bb Z}}
\def \Q {\hbox{\bb Q}}
\def \N {\hbox{\bb N}}

\def \F {\hbox{\bb F}}
\def \R {\hbox{\bb R}}

\def \Tr {\mbox{\rm{Tr}}}

\def \T {\mathcal{T}}
\def \X {\mathcal{X}}
\def \AA {\mathcal{A}}
\def \B {\mathcal{B}}

\def \H {\mathcal{H}}
\def \U {\mathcal{U}}
\def \O {\mathcal{O}}
\def \K {\mathcal{K}}
\def \Q {\hbox{\bb Q}}

\def \II {\mathcal{I}}
\def \sgn {\mbox{\rm{sgn}}}

\newcommand{\GNP}{\ensuremath{\mbox{\rm{GNP}}}}
\newcommand{\NP}{\ensuremath{\mbox{\rm{NP}}}}
\newcommand{\Ima}{\ensuremath{\mbox{\rm{Im }}}}
\newcommand{\Gal}{\ensuremath{\mbox{\rm{Gal}}}}

\author{R\'egis Blache}
\address{\'Equipe AOC,
IUFM de la Guadeloupe}
\email{rblache@iufm.univ-ag.fr}

\title[First vertices]{First vertices for generic Newton polygons}

\begin{document}

\begin{abstract}
In this paper we compute the first vertex of a generic Newton polygon in some special cases, and the corresponding Hasse polynomial. This allows us to show the nonexistence of $p$-cyclic coverings of the projective line in characteristic $p$ with supersingular jacobian for some (infinite families of) genera.
\end{abstract}

\subjclass[2000]{11L,14H}
\keywords{Character sums, $L$-functions, Newton polygons, Artin-Schreier curves}

\maketitle

\section{Introduction} 

A motivation for this paper is the study of the Torelli locus of jacobians of genus $g$ curves inside the space of principally polarized abelian varieties of dimension $g$. One approach for this problem is via Newton polygons \cite{oo2}. This seems to date back to Manin \cite{ma} and his ``formal types". The Newton polygon stratification of the space of principally polarized abelian varieties over a finite field was intensely studied by Oort and others, and much is known, in particular about the supersingular stratum, the one with the ``higher" possible polygon. Here are some open questions about the intersections of the Torelli locus with the Newton strata 
\begin{itemize}
	\item does there exist for any $g$, $p$, a curve of genus $g$ over $\overline{\F}_p$ with supersingular jacobian ?
	\item does there exist Newton strata that do not meet the Torelli locus ?
	\item the same questions with the hyperelliptic locus...
\end{itemize}
We do not pretend to answer any of these questions here, at least at this level of generality. 

\medskip

We shall study Artin Schreier curves, i.e. $p$-cyclic covering of the projective line in characteristic $p$. These curves have the interesting property that, whereas a generic curve is ordinary (i.e. has lower possible Newton polygon, i.e. has $p$-rank $g$), it is easy from Deuring Shafarevic formula to construct Artin Schreier curves with little $p$-rank \cite{pz}. Thus they are well suited to the study of the intersection of the Torelli locus with strata associated to ``high" polygon. Moreover in even characteristic, they are exactly the hyperelliptic curves. Let us recall what is known.

Van der Geer and van der Vlugt \cite{vv1} construct Artin-Schreier curves with supersingular jacobians, and from the fibered products of these curves, deduce that in characteristic $2$ there are curves of any genus with their jacobian supersingular, answering the first question above for $p=2$. In \cite{sz1}, Scholten and Zhu prove that there is no hyperelliptic curve of genus $2^n-1$, $n\geq 3$, in characteristic $2$. This result stands in striking contrast with the one of van der Geer and van der Vlugt. When $p$ is odd not much is known \cite{vv2}, \cite{sz3}; in this paper we shall show there is no $p$-cyclic covering of the projective line with supersingular jacobian and genus $g=\frac{(p-1)(d_0-1)}{2}$ when $d_0\in\{p^n-1,2p^n-2\}$, and $n(p-1)>2$.

\medskip

Let us introduce some notations. Since the numerator of the zeta function of an Artin-Schreier curve can be expressed as a product of $L$-functions associated to additive characters, we focus on these objects. Note that they have their own interest (see for instance \cite{wa}).

Consider a finite subset $D\subset \N_+$ consisting of positive integers; let $d_0=\max D$, and $k$ denote a finite field with $q:=p^m$ elements. In this paper we set
$$k[x]_D:=\left\{ f(x)=\sum_{D}a_d x^d,~a_d\in k,~a_{d_0}\in k^\times\right\}.$$

Let $k_r$ denotes the degree $r$ extension of $k$ inside a fixed algebraic closure $\overline{\F}_p$ of $k$, and for $\psi$ a non trivial additive character of $k$, denote by $\psi_r:=\psi\circ \Tr_{k_r/k}$ its extension to $k_r$. If $f$ is a degree $d_0$ polynomial over $k$, we define the exponential sums

$$S_r(f):=\sum_{x\in k_r} \psi_r(f(x)),~r\geq 1.$$

To the sequence $(S_r(f))_{r\geq 1}$ one associates the $L$-function

$$L(f;T):=\exp\left(\sum_{r\geq 1} S_r(f)\frac{T^r}{r}\right).$$

It is well known since Weil that this is a polynomial of degree $d_0-1$ in $\Z[\zeta_p][T]$, all of whose reciprocal roots are $q$-Weil numbers of weight $1$, i.e. algebraic integers all of whose conjugates have complex absolute value $q^\frac{1}{2}$. It remains to consider the question of their $p$-adic valuation. We shall give results on the Newton polygon $\NP_q(f):=\NP_q(L(f;T))$, which is very close to the one of the Artin Schreier curve $y^p-y=f(x)$.

As calculations in low degree show, this is in general a very difficult question, and we shall restrict our attention to generic Newton polygons. Assume the polynomial $f_t$ varies in a family parametrized by $t$ varying in an affine scheme $\AA/\F_p$. From Grothendieck's specialization theorem \cite{ka}, there is a Zariski dense open subset $\U$ in $\AA$ (the {\it open stratum}) and a {\it generic Newton polygon} $\GNP(\AA,p)$ such that 

\begin{itemize}
	\item[i/] for any $t\in \U(k)$, we have $\NP_q(f_t)=\GNP(\AA,p)$;
	\item[ii/] moreover, $\NP_q(f_t)\preceq \GNP(\AA,p)$ for any $t\in \AA(k)$.
\end{itemize}

\medskip

Let $\AA_D/\F_p$ denote the affine scheme parametrizing polynomials in $\overline{\F}_p[x]_D$. The aim of these notes is to determine, under a technical hypothesis, the first vertex of the polygon $\GNP(D,p):=\GNP(\AA_D,p)$. Up to an homothety of scale $p-1$ this is equivalent to the following construction. Let us define a $p$-divisible group $\X\rightarrow \AA_D$ by setting $\X_f$ the $p$-divisible group of the jacobian of the curve $y^p-y=f(x)$. Then the Newton polygon of the generic fiber $\X_\eta$ is exactly $\GNP(D,p)$. 

Moreover we show that the subset of polynomials such that $\NP_q(f)$ and $\GNP(D,p)$ have the same first vertex is the complement of an affine subscheme of $\AA_D$ defined by the vanishing of a polynomial $\H_{D,p}^1$ that we give explicitely, and call the {\it Hasse polynomial}.

\medskip

This question has already been studied, and we now recall what is known. 

When $p\geq d$, the first slope is known \cite[Theorem 1.1]{sz3}; actually when $p\geq 3d$ we know the whole generic Newton polygon \cite{bf}. The problem is much harder when $d>p$, and it includes many interesting cases, such as hyperelliptic curves in characteristic $2$. For this reason, this case has drawn much attention, and the works we are aware only treat the case $p=2$. In \cite{sz1}, Scholten and Zhu give the first slope of $\GNP(D,2)$ for any $D=\{1\leq i\leq 2n+1,~(i,2)=1\}$, and give necessary conditions on a polynomial $f\in \F_q[x]_D$, $q=2^m$, so that $\GNP(D,2)$ and $\NP_q(f)$ have the first slope. In order to do this, they compute the Verschiebung action on the first de Rham cohomology of a curve by taking power series expansions at a rational point \cite{ny}, \cite{sz3}, and they use Katz's sharp slope estimate \cite{ka}. In \cite{sz2}, they use the same method to determine the hyperelliptic curves with supersingular jacobian in characteristic $2$ for any genus $g\leq 8$.

\medskip

In \cite{bl}, the author gives the first slope of any polygon $\GNP(D,p)$, building on the work of Moreno et al. \cite{mm}. This work relies on a lower bound for the valuation of the exponential sums $S_r(f)$, $f\in k[x]_D$, $r\geq 1$, and the existence of polynomials for which the bound is attained. But it does not say anything about which polynomials attain this bound. We precise this work here, giving the first vertex in some cases, and the Hasse polynomial. 

Our method is very close to Dwork's original one: we shall consider precisely the matrix $A_m$ of Dwork's operator (associated to our situation), acting on a space of overconvergent functions. Actually the approach is mainly inspirated by Robba's paper \cite{ro}. It relies on an estimate of the valuation of the coefficients of Dwork's splitting functions in Lemma \ref{coeffssplit}. This estimate allows us to give a link between with the terms of minimal valuation of the principal minors of $A_m$ and certain solutions (called minimal) of modular equations associated to $D$ and $p$, as in \cite{mm}. Since these minors are the building blocks for the coefficients of the $L$-function, we get a congruence on the last coefficient of the $L$-function on the first slope, giving the last vertex and corresponding Hasse polynomial.

Even if it depends on a technical hypothesis, this result allows us to find in another way the known cases: Theorem 1.1 in \cite{sz3}, and $D=\{1\leq i\leq 2n+1,~(i,2)=1\}$ in characteristic $2$. But the main interest is to give such results in odd characteristic $p$, for $D=\{1\leq i\leq d_0\},~(i,p)=1\}$ when $p^n-1\leq d_0\leq 2p^n-2$. We deduce that for $d_0\in\{p^n-1,2p^n-2\}$, all Newton polygons associated to polynomials of degree $d_0$ have the same first slope $\frac{1}{n(p-1)}$.

\medskip

After a brief review of the main results from $p$-adic cohomology that we need here, we recall the necessary material from \cite{bl} and \cite{mm} in section 2, and give the link with the valuation of minors in section 3. Under a technical condition, we get in this way a congruence for the last coefficient of the $L$ function on the first slope in Theorem \ref{mainth}. In section 4, we give applications of this Theorem, and deduce the nonexistence of supersingular Artin-Schreier curves for some (infinite families of) genera.

\section{Cohomology}

Here we briefly describe the cohomological tools that we shall use in the paper. Everything could be described in terms of rigid cohomology, but for sake of simplicity we shall adopt the point of view of Robba \cite{ro}, which has the benefit to be explicit.

\medskip

We denote by $\Q_p$ the field of $p$-adic numbers, and by $\K_m=\Q_p(\zeta_{q-1})$ its 
(unique up to isomorphism) unramified extension of degree $m$. Let 
$\O_m=\Z_p[\zeta_{q-1}]$ be the valuation ring of $\K_m$; the elements of finite order in 
$\O_m^\times$ form a group $\T_m^\times$ of order $p^m-1$, and 
$\T_m:=\T_m^\times\cup\{0\}$ is the {\it Teichm\"uller} of $\K_m$. Note that 
it is the image of a section of reduction modulo $p$ from $\O_m$ to its 
residue field $\F_q$, called the {\it Teichm\"uller lift}. Let $\tau$ be the Frobenius; it is the 
generator of $\Gal(\K_m/\Q_p)$ which acts on $\T_m$ as the $p$th power 
map. Finally we denote by $\C_p$ a completion of a fixed algebraic 
closure $\overline{\Q}_p$ of $\Q_p$.

\medskip

Let $\pi \in \C_p$ be a root of the polynomial $X^{p-1}+p$. It is well known that $\Q_p(\pi)=\Q_p(\zeta_p)$ is a totally ramified extension of degree $p-1$ of $\Q_p$. We shall 
frequently use the valuation $v:=v_\pi$, normalized by $v_\pi(\pi)=1$, 
instead of the usual $p$-adic valuation $v_p$, or the $q$-adic valuation 
$v_q$. 

\medskip

For any $\omega \in \C_p$, $r\in \R$, we denote by $B(\omega,r^+)$ the closed  ball in $\C_p$ with center $\omega$ and radius $r$. Let $A:=B(0,1^+)$. We consider the space $\H^\dagger(A)_0$ of overconvergent 
analytic functions on $A$ with no constant term.

\medskip

We define the power series $\theta(X):=\exp(\pi X-\pi X^p)$; this is a {\it splitting function} in Dwork's terminology ({\it cf.} \cite{dw} p55). Its values at the points of $\T_1$ are $p$-th roots of unity; in other words this function represents an additive character of order $p$. It is well known that 
$\theta$ converges for any $x$ in $\C_p$ such that 
$v_p(x)>-\frac{p-1}{p^2}$, and in particular $\theta \in \H^\dagger(A)$. Let 
$$\theta_m(X):=\prod_{i=0}^{m-1} \theta(X^{p^i})=\exp(\pi X-\pi X^q).$$

If $f(X):=\alpha_dX^d+\dots+\alpha_1X$, $\alpha_d\neq 0$ is a 
polynomial of degree $d$, prime to $p$, over $k$, we denote by 
$\widetilde{f}(x):=a_dX^d+\dots+a_1X \in \O_m[X]$ the polynomial whose 
coefficients are the Teichm\"uller lifts of those of $f$.

\medskip

We define the functions 
$$\begin{array}{rcl}
F_1(X) & := & \prod_{i=1}^d  \theta(a_iX^i):=\sum_{n\geq0} f^{(1)}_nX^n,\\
F_m(X) & := & \prod_{i=0}^{m-1} F^{\tau^i}(X^{p^i})=\prod_{i=1}^d \theta_m(a_iX^i):=\sum_{n\geq0} f^{(m)}_nX^n;\\
\end{array}$$
since $\theta$ is overconvergent, $F_1$ and $F_m$ 
also, and we get $F_m\in \H^\dagger(A)$. 

\medskip

Consider the mapping $\psi_q$ defined on $\H^\dagger(A)$ by 
$\psi_q h(x):=\frac{1}{q}\sum_{z^q=x}h(z)$; if $h(X)=\sum b_nX^n$, then 
$\psi_q h(X)=\sum b_{qn}X^n$. Let $\alpha_m:=\psi_q \circ F_m$; this is an operator over $\H^\dagger(A)_0$, and Dwork's trace 
formula gives the following (cf \cite{ro}, p235)
$$L(f,T)=\frac{\det(1-T\alpha_m)}{\det(1-qT\alpha_m)}.$$

Since we are looking for the first slope of a Newton polygon in a one dimensional situation, it is less than or equal to $\frac{1}{2}$ and we just have to consider the first slope of the Newton polygon of $\det(1-T\alpha_m)$. In the basis $\B(X,X^2,\cdots)$ of $\H^\dagger(A)_0$, the matrix of $\alpha_m$ is given by 
$$A_m=(f^{(m)}_{qi-j})_{i,j\geq 1}.$$

Recall the factorization of $\alpha_m$. Let $\alpha_1$ be the endomorphism of $\H^\dagger(A)_0$ defined by 
$\alpha_1=\psi_p\circ F_1$; we have $\alpha_m=\alpha_1^{\tau^{m-1}}\dots \alpha_1^{\tau}\alpha_1$. Thus if $A_1$ is the matrix of $\alpha_1$ with respect to the basis $\B$, we have
$$A_m=A_1^{\tau^{m-1}}\cdots A_1,\qquad A_1=(f^{(1)}_{pi-j})_{i,j\geq 1}.$$

\section{Some technical results and tools}

The aim of this section is to recall some facts and notations about solutions of certain modular equations. The ideas come from work of Moreno, Kumar, Castro and Shum \cite{mm}. 

We begin by giving a precise estimate for the valuations of the coefficients of the series $\theta_m(X):=\sum_{n\geq 0} \lambda^{(m)}_n X^n$, then we recall facts from \cite{bl} on the $p$-density of the set $D$.

\subsection{Coefficients of the splitting function}

\begin{definition}
For $n$ a non negative integer, we denote by $s_p(n)$ the $p$-weight of $n$ : in other words, if $n=n_0+pn_1+\cdots+p^tn_t$ with $0\leq n_i\leq p-1$, we have $s_p(n)=n_0+\dots+n_t$.
Moreover we set $n!!:=n_0!\dots n_{t-1}!$.
\end{definition}

We have (compare Stickelberger's theorem on Gauss sums \cite{st})

\begin{lemma}
\label{coeffssplit}
Write the integer $n$ as in the definition above. In the ring $\Z_p[\zeta_p]$, we have
\begin{itemize}
	\item[i/] $\lambda^{(m)}_n\equiv \frac{\pi^{s_p(n)}}{n!!} \mod \pi^{s_p(n)+p-1}$ if $0\leq n\leq q-1$;
	\item[ii/] $v(\lambda^{(m)}_n)\geq s_p(n)+p-1$ if $n\geq q$.
\end{itemize}
\end{lemma}

\begin{proof}
Recall that $\theta_m(X)=\exp(\pi X-\pi X^q)$. From the well known expansion $\exp X=\sum_{n\geq 0} \frac{X^n}{n!}$, we get
$$\lambda^{(m)}_n=\sum_{r,s,r+qs=n} (-1)^s \frac{\pi^{r+s}}{r!s!}.$$

Assume $0\leq n\leq q-1$; then we get $\lambda^{(m)}_n=\frac{\pi^{n}}{n!}$. From a result of Anton, we have the congruence 
$$n!\equiv (-p)^a n!! \mod p^{a+1},\quad a=\frac{n-s_p(n)}{p-1},$$
which gives part i/ of the lemma (recall that $\pi^{p-1}=-p$).

Assume $n\geq q$, and write $n=n_0+pn_1+\cdots+p^tn_t$, with $t\geq m$. First observe that, from above,
$$\frac{\pi^{r+s}}{r!s!}\equiv \frac{\pi^{s_p(r)+s_p(s)}}{r!! s!!} \mod \pi^{s_p(r)+s_p(s)+p-1}$$
Since $n=r+qs$, and $s_p(q)=1$, we remark that $s_p(n)\leq s_p(r)+s_p(s)$, and $s_p(n)\equiv s_p(r)+s_p(s) \mod p-1$. We have $s_p(n)= s_p(r)+s_p(s)$ if and only if $0\leq s_i \leq n_{m+i}$ for any $0\leq i\leq t-m$. Thus we get
$$\lambda^{(m)}_n\equiv \left(\sum_{s_0=0}^{n_m}\cdots \sum_{s_{t-m}=0}^{n_t} (-1)^s\frac{1}{r!!s!!}\right) \pi^{s_p(n)}\mod \pi^{s_p(n)+p-1}.$$

Assume $p$ is odd. Since $(-1)^s=(-1)^{s_0+\cdots+s_{t-m}}$, and $r!!=n_0!!\dots n_{m-1}!!(n_m-s_0)!!\cdots (n_t-s_{m-t})!!$, one can rewrite the multiple sum
$$\frac{1}{n!!}\sum_{s_0=0}^{n_m}\frac{(-1)^{s_0}n_m!}{(n_m-s_0)!s_0!}\cdots \sum_{s_{t-m}=0}^{n_t} \frac{(-1)^{s_{t-m}}n_t!}{(n_t-s_{t-m})!s_{t-m}!}.$$
Since each sum is an alternate sum of binomial coefficients, it is zero, and we get the result.

\medskip

When $p=2$, we just have to remove the signs; now the sums of binomial coefficients are powers of $2$, and we get the result.
\end{proof}

\subsection{Modular equations and {\it p}-density} 

Here we recall and introduce some notations that we shall use in the following. The reader interested in more details and the proofs can refer to \cite{bl}.

\begin{definition}
Let $D$ be as above, and $m$ denote a positive integer.

\begin{itemize}
	\item[\it i/] We define $E_{D,p}(n)$ as the set of $\# D$-tuples $U=(u_d)_{d\in D}\in \{0,\dots,p^n-1\}^{\# D}$ such that
$$\sum_D u_d d \equiv 0~[p^n-1],~\sum_D u_d d>0.$$
For any $U \in E_{D,p}(n)$, we define the {\rm $p$-weight of $U$} as the integer $s_p(U)= \sum_{d\in D} s_p(u_d)$, and the {\rm length of $U$} as $\ell(U)=n$.
\item[\it ii/] Define $s_{D,p}(n):=\min_{U\in E_{D,p}(n)} s_p(U)$.
\item[\it iii/] Let $\delta_n$ be the {\rm shift}, from the set $\{0,\dots,p^n-1\}$ to itself, which sends any integer $0\leq k\leq p^n-2$ to the residue of $pk$ modulo $p^n-1$, and $p^n-1$ to itself. 
\item[\it iv/] We define a map
$$\begin{array}{ccccc}
\varphi_n & : & E_{D,p}(n) & \rightarrow & \N_+ \\
        &   & U      & \mapsto & \frac{1}{p^n-1} \sum_{d\in D} du_d\\
        \end{array}$$
\item[\it v/] To each $U \in E_{D,p}(n)$, we associate a map $\varphi_U$ from $\{0,\dots,n-1\}$ to $\N_+$ defined by $$\varphi_U(k):=\varphi_n(\delta_n^k(U)),$$
and we denote its image by $\Phi(U)$; we say $U$ is {\rm irreducible} when $\#\Phi(U)=n$, i.e. when $\varphi_U$ is an injection.
\end{itemize}
\end{definition}  

Moreno et al. \cite{mm} introduce the sets $E_{D,p}(n)$ in order to give a lower bound for the valuation of an exponential sum associated to a polynomial with its exponents in $D$ and coefficients in $\F_{p^n}$. In \cite{bl}, we proved the following (see Proposition 1.1 and Lemma 1.5) 

\begin{proposition}
\label{princ}
The set $\left\{\frac{s_{D,p}(n)}{n}\right\}_{n\geq 1}$ has a minimum; this minimum is attained for at least one $n\leq d_0-1$.
\end{proposition}

This result allows the definition of $p$-density, that we shall use in the next sections

\begin{definition}
\begin{itemize}
	\item[\it i/] Let $D,p$ be as above. The {\rm $p$-density of the set $D$} is the rational number
$$\delta_{D,p}:= \frac{1}{p-1}\min_{n\geq 1}\left\{\frac{s_{D,p}(n)}{n}\right\}.$$
\item[\it ii/] The density of an element $U\in E_{D,p}(n)$ is $\delta(U):=\frac{s_p(U)}{(p-1)n}$. The element $U$ is {\rm minimal} when $\delta(U)=\delta_{D,p}$.
\end{itemize}
\end{definition}

We shall need the following lemma

\begin{lemma}
\label{qmin}
Let $(u_d)_D$ be nonnegative integers such that $\sum_D u_d d \equiv 0~[p^n-1],~\sum_D u_d d>0$. Then we have $\sum_D s_p(u_d)\geq s_{D,p}(n)$.
\end{lemma}

\begin{proof}
Let $u$ be a positive integer, and $\overline{u}\in \{1,\cdots,p^n-1\}$ be the integer defined by $u\equiv \overline{u}$ mod $p^n-1$. Then we have $\sum_D \overline{u}_d d \equiv 0~[p^n-1],~\sum_D \overline{u}_d d>0$ from the construction. Now we have $\sum_D s_p(\overline{u}_d)\geq s_{D,p}(n)$ from the definition of this last invariant, and the result comes from the following lemma.
\end{proof}

\begin{lemma}
\label{qred}
Notations being as in the proof above, we have $s_p(u)\geq s_p(\overline{u})$. 
\end{lemma}

\begin{proof}
Write the euclidean division of $u$ by $p^n$, $u=p^nu_1+v_1$. Then we have $s_p(u)=s_p(u_1)+s_p(v_1)\geq s_p(u_1+v_1)$. Replacing $u$ by $u_1+v_1$, and repeating the same process, we finally get $\overline{u}$ and the result.
\end{proof}

Finally, we introduce some notations and sets for further use

\begin{definition} 
We denote by $MI_{D,p}(n)$ the set of minimal irreducible elements in $E_{D,p}(n)$, and we set
$$MI_{D,p}:=\coprod_n MI_{D,p}(n).$$ 
Moreover, if $MI_{D,p}=\{U_1,\dots,U_r\}$, we set
$$\Sigma_{D,p}=\cup_{i=1}^r \Phi(U_i);~N_{D,p}:=\#\Sigma_{D,p}.$$
Since $\delta_n$ preserves both minimality and irreducibility for the elements in $E_{D,p}(n)$, it acts on $MI_{D,p}(n)$. We define $\Delta$ as the bijection over $MI_{D,p}$ defined from the $\delta_n$, and we denote by $MI_{D,p}^\ast$ the set of its orbits.
\end{definition}

\begin{remark}
An element $U\in \N^{\# D}$ can belong simultaneously to various sets $E_{D,p}(n)$ when $n$ varies, but a minimal element belongs to exactly one set $MI_{D,p}(n)$. For this reason, the disjoint union makes sense, so as the definition of $\Delta$.
\end{remark}

\begin{remark}
Note that $\varphi$ as a finite image from \cite[Lemma 1.1 iii/]{bl}. If $N_D$ denotes its cardinality, we get that an element in $E_{D,p}(n)$ can be irreducible only if $n\leq N_D$. Thus the set $MI_{D,p}$ is finite, so as $MI_{D,p}^\ast$.
\end{remark}

\begin{remark}
\label{minelt}
From the definition, for $U\in E_{D,p}(n)$ we have $\varphi_{\delta_n(U)}(i)=\varphi_U(i+1)$, that is $\varphi_{\delta_n(U)}$ is obtained by $\varphi_U$ by a cyclic permutation of its image. Thus in each orbit of the action of $\delta_n$ on $MI_{D,p}(n)$ there exists a unique $U$ (since $U$ is irreducible, $\varphi_U$ is injective) such that $\varphi_U(0)=\min \Ima \varphi_U$. We will choose this element to represent the corresponding orbit in the following.
\end{remark}

We end with a lemma that we shall use further

\begin{lemma} 
\label{suppmin}
Let $U\in E_{D,p}(n)$ be minimal. Then $\Phi(U)\subset \Sigma_{D,p}$.
\end{lemma}

\begin{proof}
If $U$ is irreducible, this is clear from the definition of $\Sigma_{D,p}$. Else we can find integers $t_1< t_2$ in $\{0,\dots,n-1\}$ such that $\varphi(\delta_n^{t_1}(U))=\varphi(\delta_n^{t_2}(U))$. If we consider the element $\delta_n^{t_1}(U)$ instead of $U$ (they have the same $p$-weight), we obtain some $0<t\leq n-1$ such that $\varphi(U)=\varphi(\delta_n^{t}(U))$.

For each $d$, let $u_d=p^{n-t}w_d+v_d$ be the result of the euclidean division of $u_d$ by $p^{n-t}$. Set $V=(v_d)$, and $W=(w_d)$. From \cite[Lemma 1.2 ii/]{bl} and the definition of $t$, we have
$$\sum_{D} dv_d=(p^{n-t}-1)\varphi(U)~;~\sum_{D} dw_d=(p^{t}-1)\varphi(U).$$
Thus $V\in E_{D,p}(n-t)$ and $W\in E_{D,p}(t)$. From the definition of $p$-density, both $\delta(V)$ and $\delta(W)$ are greater than or equal to $\delta_{D,p}$. But for each $d$ we have $s_p(u_d)=s_p(v_d)+s_p(w_d)$, and $s_p(U)=s_p(V)+s_p(W)$. Since $U$ is minimal, we have 
$$\delta_{D,p}=\frac{s_p(V)+s_p(W)}{n(p-1)}=\left(1-\frac{t}{n}\right)\delta(V)+\frac{t}{n}\delta(W)\geq \delta_{D,p}.$$ 
Thus both $V$ and $W$ are minimal, and again from \cite[Lemma 1.2 ii/]{bl}, we have $\Phi(U)=\Phi(V)\cup\Phi(W)$. If both $V$ and $W$ are irreducible, we are done; else we apply the same process to $V$ or $W$, and we end with minimal irreducible elements $U_{i_1},\cdots,U_{i_k}$ in $E_{D,p}(n_1),\cdots , E_{D,p}(n_k)$ with $\Phi(U)=\cup_j \Phi(U_{i_j})$.
\end{proof}

\section{The first vertex of the generic Newton polygon.}

Recall that $A_m=(f^{(m)}_{qi-j})_{i,j\geq 1}$ ({\it resp.} $A_1=(f^{(1)}_{pi-j})_{i,j\geq 1}$) is the matrix of $\alpha_m$ ({\it resp.} $\alpha_1$) with respect to the basis $\B$, and their coefficients are in $\Q_p(\zeta_p,\zeta_{q-1})$. Moreover recall the relation

$$A_m=A_1^{\tau^{m-1}}\cdots A_1.$$

Set $\det(I-TA_m):=1+\sum_{n\geq 1} \ell^{(m)}_n T^n$, and $\det(I-TA_1):=1+\sum_{n\geq 1} \ell^{(1)}_n T^n$.

\medskip

In order to simplify notations, we shall set $\delta:=\delta_{D,p}$, $\Sigma:=\Sigma_{D,p}$ and $N:=N_{D,p}$ all along this section. Moreover we set $\Sigma:=\{v_1,\cdots,v_N\}$.

\medskip

Our aim is to get a congruence for the coefficient $\ell_N^{(m)}$ under a technical hypothesis, and to deduce from this congruence the first vertex of the generic Newton polygon, and the corresponding Hasse polynomial.

We begin by recalling some facts about the coefficients $\ell_n^{(m)}$: we shall decompose their principal parts as sums of terms that we link with to the minimal elements defined in the preceding section.

Let $\II_n$ denote the set of injections from $\{0,\cdots,n-1\}$ to $\N_+$, and $S_n$ the symmetric group over $n$ elements. 

Let $F$ be a non empty subset of $\N_+$. We denote by $A_m^F$ the matrix $(f^{(m)}_{qi-j})_{i,j\in F}$. We have the following expression for $\ell_n^{(m)}$ in terms of the determinants of the matrices $A_m^F$
$$\ell_n^{(m)}=\sum_{F\subset \ma{N}_+,~\# F=n} \det A_m^F.$$
From the definition of the determinant, we have, for $F=\{u_0,\dots,u_{n-1}\}$
$$\det A_m^F=\sum_{\sigma\in S_n} M_{F,\sigma},\quad M_{F,\sigma}:=\sgn(\sigma)\prod_{i=0}^{n-1} f^{(m)}_{qu_i-u_{\sigma(i)}},$$

We will need another, less classical expression for the determinant in the following. Let us give a definition

\begin{definition}
\label{isigma}
Let $F$ be as above, with cardinality $n$. Define $\II(F):=\coprod_{k=1}^n \II_k(F)$, where $\II_k(F)$ is the set of injections from $\{0,\cdots,k-1\}$ to $F$,. Let $\AA(F)$ consist of the parts $\Theta:=\{\theta_1,\dots,\theta_{k_\Theta}\} \subset \II(F)$ such that 
\begin{itemize}
	\item[i/] for each $i$, $\theta_i(0)=\min \Ima \theta_i$; 
	\item[ii/] the $\Ima \theta_i$, $1\leq i\leq k_\Theta$ form a partition of $F$.
\end{itemize}
\end{definition}

From this new set, we have the following expression for the determinant $\det A_m^F$

\begin{lemma}
\label{decdetcyc}
For $\theta\in \II_n$, set $M_\theta^{(m)}:=(-1)^{n-1}\prod_{i=0}^{n-1} f^{(m)}_{q\theta(i)-\theta(i+1)}$. Notations being as in the definition above, we have

$$\det A_m^F = \sum_{\Theta \in \AA(F)} \prod_{i=1}^{k_\Theta} M_{\theta_i}^{(m)}.$$
\end{lemma}

\begin{proof}
First recall that any $\sigma \in S_n$ can be written in a unique way as $\sigma=\gamma_1\cdots \gamma_{k_\sigma}$ where the $\gamma_i$ are cycles with pairwise disjoint supports covering $\{0,\cdots,n-1\}$. Such a cycle $\gamma_i$ of length $n_i$ can be represented in a unique way as an injection $\eta_i$ from $\{0,\cdots,n_i-1\}$ to $\{0,\cdots,n-1\}$ such that $\eta_i(0)=\min \Ima \eta_i$ in the following way: $\gamma_i=(\eta_i(0)\cdots \eta_i(n_i-1))$. Thus the map $\sigma \mapsto \{\eta_1,\cdots,\eta_{k_\sigma}\}$ defines a bijection between the sets $S_n$ and $\AA(\{0,\cdots,n-1\})$.

Let $g$ be the bijection from $\{0,\cdots,n-1\}$ to $F$ sending $i$ to $u_i$; from what we have just said, the map $\sigma \mapsto \{g\circ \eta_1,\cdots,g \circ \eta_{k_\sigma}\}$ is a bijection from $S_n$ to $\AA(F)$. Now for any $\sigma\in S_n$ with image $\{\theta_1,\cdots,\theta_{k_\sigma}\}$ in $\AA(F)$, we have $M_{F,\sigma}=\prod_{i=1}^{k_\sigma} M_{\theta_i}^{(m)}$. This is the desired result.
\end{proof}

Before we give congruences for the cyclic minors, recall that we have the following expansion for the coefficients of the series $F_m$
\begin{equation}
\label{coeffm}
f_n^{(m)}=\sum_{\sum du_d=n} \prod_D \lambda_{u_d}^{(m)} a_d^{u_d}
\end{equation}

We begin with a lemma

\begin{lemma}
\label{supp1}
Notations being as above, we have
\begin{itemize}
	\item[i/] $M_\theta^{(m)}\equiv 0 \mod \pi^{mn(p-1)\delta}$, and
	\item[ii/] if $\Ima \theta \nsubseteq \Sigma$, then $M_\theta^{(m)}\equiv 0 \mod \pi^{mn(p-1)\delta+1}$
\end{itemize}
\end{lemma}

\begin{proof}
From the expression of $M_\theta^{(m)}$ and (\ref{coeffm}), we have that $M_\theta^{(m)}$ is a sum of terms like
$$A_{(u_d^{i})}= (-1)^{n-1}\prod_{i=0}^{n-1} \prod_D \lambda_{u^{i}_d}^{(m)} a_d^{u^{i}_d},$$

where for each $i$ we have $\sum_D du_d^{i}=q\theta(i)-\theta(i+1)$. From Lemma \ref{coeffssplit}, the valuation of such a term satisfies

\begin{equation}
\label{valq}
v(A_{(u_d^{i})})\geq \sum_{i=0}^{n-1} \sum_D s_p(u_d^{i}),
\end{equation}
with equality if and only if we have $0\leq u_d^{i}\leq q-1$ for any $i,d$.

For each $d$, set $u_d=\sum_{i=0}^{n-1} q^{n-1-i} u_d^{i}$. A rapid calculation gives the equality
\begin{equation}
\label{equa}
\sum_D du_d=(q^n-1)\theta(0).
\end{equation}
By the way we defined the integers $u_d$, we have $\sum_{i,d} s_p(u_d^{(i)})\geq \sum_D s_p(u_d)$, and from Lemma \ref{qmin}, we have $\sum_D s_p(u_d)\geq mn(p-1)\delta$. Together with equation (\ref{valq}), this proves assertion i/.

Assume that $v(A_{(u_d^{i})})=mn(p-1)\delta$. Then the three inequalities above are equalities. But equality in (\ref{valq}) implies the second equality, and that $0\leq u_d\leq p^{mn-1}$ for each $d$. Thus $U:=(u_d)$ is an element of $E_{D,p}(mn)$ from (\ref{equa}). Equality for the third gives that $U$ is a minimal element in $E_{D,p}(mn)$. Finally, from the definition of $u_d$ and \cite[Lemma 1.2 ii/]{bl}, we have for any $i$ that
$$\varphi_U(mi)=\theta(i),~0\leq i \leq n-1,$$
and $\Ima \theta \subset \Phi(U)$. The second assertion follows from Lemma \ref{suppmin}.
\end{proof}

We are ready to gie a first congruence for the coefficient $\ell_N^{(m)}$.

\begin{lemma}
\label{supp2}
Let $F\subset\N_+$, with cardinality $n$. Then 
\begin{itemize}
	\item[i/] if $F\nsubseteq \Sigma$, then $\det A_m^F\equiv 0 \mod \pi^{mn(p-1)\delta+1}$;
	\item[ii/] we have $\ell_N^{(m)}\equiv \det A_m^\Sigma \mod \pi^{mN(p-1)\delta+1}$.
\end{itemize}
\end{lemma}

\begin{proof}
Recall from Lemma \ref{decdetcyc} that we can write $\det A_m^F$ as a sums of terms like $M_{\theta_1}^{(m)}\cdots M_{\theta_k}^{(m)}$, with the $\Ima \theta_i$ pairwise disjoint and $\cup\Ima \theta_i =F$.

If $F\nsubseteq \Sigma$, there exists at least one $\theta_i$ such that $v(M_{\theta_i}^{(m)}) \geq  mn_i(p-1)\delta+1$, where we have set $n_i:=\# \Ima \theta_i$. Applying Lemma \ref{supp1} i/ to the $M_{\theta_j}^{(m)}$, $j\neq i$, we get assertion i/ since $n=n_1+\cdots+n_k$.

Assertion ii/ is an easy consequence of the first one, and the expression of $\ell_N^{(m)}$ as a sum of principal minors.
\end{proof}

We shall now use the factorisation of the matrix $A_m$ in terms of $A_1$. As usual we begin by considering cyclic minors. From Lemma \ref{supp1}, we choose some injection $\theta$ whose image is contained in $\Sigma$, and let $M_\theta^{(m)}$ be as above. From \cite[Lemma 3.2]{bf} and the factorization of $A_m$ in terms of $A_1$ we can write

$$M_\theta^{(m)}=\sum_{(\theta_1,\cdots,\theta_{m-1})\in \II_n^{m-1}} \prod_{i=0}^{n-1} \prod_{j=0}^{m-1} \left(f_{p\theta_j(i)-\theta_{j+1}(i)}^{(1)}\right)^{\tau^{m-1-j}}.$$

where for each $i$, $\theta_0(i)=\theta(i)$ and $\theta_m(i)=\theta(i+1)$.

Now we have 

\begin{lemma}
\label{supp2bis}
Notations are as above. Assume that for some $j$ we have $\Ima\nsubseteq \Sigma$; then we have the congruence
$$\prod_{i=0}^{n-1} \prod_{j=0}^{m-1} \left(f_{p\theta_j(i)-\theta_{j+1}(i)}^{(1)}\right)^{\tau^{m-1-j}} \equiv 0 \mod \pi^{mN(p-1)\delta+1}.$$
\end{lemma}

\begin{proof}
We have the following for the coefficients of the series $F_1$

$$f_n^{(1)}=\sum_{\sum du_d=n} \prod_D \lambda_{u_d}^{(1)} a_d^{u_d}.$$

 Thus $M_\theta^{(m)}$ can be written as a sum of terms of the form

$$A_{(u_d^{ij})}= \prod_{i=0}^{n-1} \prod_{j=0}^{m-1}\left(\prod_D \lambda_{u^{ij}_d}^{(1)} a_d^{u^{ij}_d}\right)^{\tau^{m-1-j}},$$

where for each $i,j$ we have $\sum_D du_d^{ij}=p\theta_j(i)-\theta_{j+1}(i)$. From Lemma \ref{coeffssplit}, its valuation satisfies

$$v(A_{(u_d^{ij})})\geq \sum_{i=0}^{n-1} \sum_{j=0}^{m-1}\sum_D s_p(u_d^{ij}),$$
with equality if and only if we have $0\leq u_d^{ij}\leq p-1$ for any $i,j,d$. 

For each $d$, set $u_d=\sum_{i=0}^{n-1} q^{n-1-i} \sum_{j=0}^{m-1} p^{m-1-j} u_d^{ij}$. As in the proof of Lemma \ref{supp1}, we have
$$\sum_D du_d=(q^n-1)\theta(0).$$
Moreover, the equality $v(A_{(u_d^{ij})})=mn(p-1)\delta$ implies that $U:=(u_d)$ is a minimal element in $E_{D,p}(mn)$, and since $\theta_j(i)=\varphi_U(mi+j)$ (from \cite[Lemma 1.2 ii/]{bl}), we have $\Ima \theta_j\subset \Sigma$ from Lemma \ref{suppmin}. As a consequence, if for some $0\leq j\leq m-1$, we have $\Ima \theta_j \nsubseteq \Sigma$, then 
$$M_\theta^{(m)} \equiv 0 \mod \pi^{mn(p-1)\delta+1}.$$
\end{proof}

The following lemma allows us to reduce the congruence for $\ell_N^{(m)}$ to a congruence for $\ell_1^{(m)}$. 

\begin{lemma}
\label{supp3}
Notations being as above, we have the congruence
$$\det A_m^\Sigma\equiv N_{\K_m(\zeta_{p})/\ma{Q}_p(\zeta_p)}\left(\det A_1^\Sigma\right) \mod \pi^{mN(p-1)\delta+1}.$$
\end{lemma}

\begin{proof}
Recall that we have set $\Sigma=\{v_1,\cdots,v_N\}$. Using again \cite[Lemma 3.2]{bf}, we have
$$\det A_m^\Sigma = \sum_{\sigma\in S_N} M_{\Sigma,\sigma},~M_{\Sigma,\sigma}=\sum_{\theta_1,\cdots,\theta_{m-1}\in \II_N^{m-1}}   \prod_{i=0}^{N-1} \prod_{j=0}^{m-1} \left(f_{p\theta_j(i)-\theta_{j+1}(i)}^{(1)}\right)^{\tau^{m-1-j}}.$$

where for each $i$, we set $\theta_0(i)=v_{i}$ and $\theta_m(i)=v_{\sigma(i)}$. Following the beginning of the proof and the expression of $M_{\Sigma,\sigma}$ as a product of cyclic minors, we see that the only terms appearing modulo $\pi^{mN(p-1)\delta+1}$ are those with $\Ima \theta_j=\Sigma$ for each $j$. 

For each $1\leq j\leq m-1$, define $\sigma_j$ as the unique element in $S_N$ such that $\theta_j:=\theta_0\circ \sigma_j$. We can rewrite
$$\det A_m^\Sigma  \equiv  \sum_{\sigma_0,\cdots,\sigma_{m-1}\in S_N}   \prod_{i=0}^{N-1} \prod_{j=0}^{m-1} \left(f_{p\theta_0\circ\sigma_j(i)-\theta_0\circ\sigma_{j+1}(i)}^{(1)}\right)^{\tau^{m-1-j}}\mod \pi^{mN(p-1)\delta+1}$$
and we can now exchange the product over $j$ and the sum to obtain
$$\det A_m^\Sigma  \equiv  \prod_{j=0}^{m-1} \left(\sum_{\sigma_j\in S_N} \prod_{i=0}^{N-1}  f_{p\theta_0(\sigma_j(i))-\theta_0(\sigma_{j+1}(i))}^{(1)}\right)^{\tau^{m-1-j}} \mod \pi^{mN(p-1)\delta+1}$$

Finally we have $\prod_{i=0}^{N-1}  f_{p\theta_0(\sigma_j(i))-\theta_0(\sigma_{j+1}(i))}^{(1)}=\prod_{i=0}^{N-1}  f_{p\theta_0(i)-\theta_0(\sigma(i))}^{(1)}$ for $\sigma=\sigma_{j+1}\circ \sigma_j^{-1}$, and we obtain

$$\det A_m^\Sigma \equiv  \prod_{j=0}^{m-1} \left(\sum_{\sigma\in S_N} \prod_{i=0}^{N-1}  f_{p\theta_0(i)-\theta_0(\sigma(i))}^{(1)}\right)^{\tau^{m-1-j}} \mod \pi^{mN(p-1)\delta+1}$$
This is what we wanted to show.
\end{proof}

We are reduced to give a congruence for $\ell_N^{(1)}$, i.e. from Lemma \ref{supp2} ii/, a congruence for $\det A_1^{\Sigma}$. As usual we begin with the terms $M_\theta$.

\begin{lemma}
\label{congl1}
Let $\theta$ be an injection from $\{0,\cdots,n-1\}$, $n\leq N$, to $\Sigma$. Then we have
$$M_\theta^{(1)}\equiv (-1)^{n-1}\left(\sum_{MI_{D,p}(\theta)} \frac{A^U}{U!!}\right) \pi^{n(p-1)\delta} \mod \pi^{n(p-1)\delta+1}$$ 
where $MI_{D,p}(\theta)$ consists of minimal irreducible elements $U\in MI_{D,p}(n)$ such that $\varphi_U=\theta$, and we have set $A^U:=\prod_D a_d^{u_d}$, $U!!=\prod_D u_d!!$.
\end{lemma}

\begin{proof} Recall from the proof of lemma \ref{supp1} that $M_\theta$ can be written as a sum of terms

$$A_{(u_d^{i})}= (-1)^{n-1}\prod_{i=0}^{n-1} \prod_D \lambda_{u^{i}_d}^{(m)} a_d^{u^{i}_d},$$

where for each $i$ we have $\sum_D du_d^{i}=p\theta(i)-\theta(i+1)$. For each $d$, set $u_d=\sum_{i=0}^{n-1} p^{n-1-i} u_d^{i}$. Such a term has $\pi$-adic valuation $n(p-1)\delta$ if and only if $U=(u_d)$ is a minimal element in $E_{D,p}(n)$. Moreover, from the definition of $u_d$ and \cite[Lemma 1.2 ii/]{bl}, we have that for any $i$
$$\varphi_U(i)=\theta(i),~0\leq i \leq n-1,$$
i.e. we have $\varphi_U=\theta$, and $U\in MI_{D,p}(\theta)$. 

Conversely, if $U\in MI_{D,p}(\theta)$, we get a term $A_{(u_d^{i})}$ from the digits of the $p$-ary expansions of the $u_d$, which appears in $M_\theta$ with the correct valuation. It just remains to use the congruence in Lemma \ref{coeffssplit} i/ to get the announced result.
\end{proof}

We are ready to give the congruence for $\ell_N^{(1)}$, but we first need a definition

\begin{definition}
Let $\AA_{D,p}$ be the set consisting of the parts of $\U=\{U_1,\dots,U_{k_\U}\}\subset MI_{D,p}^\ast$ such that
$$\Sigma:=\coprod_{t=1}^{k_\U} \Phi(U_{t});$$
For $\U=\{U_1,\dots,U_{k_\U}\}\in \AA_{D,p}$, we set $|\U|:=\sum_{t=1}^{k_\U} (\ell(U_t)-1)=N-k_\U$, and finally we define the polynomial $\H_{D,p}^1\in \Z_p[(a_d)_{d\in D}]$ by
$$\H_{D,p}^1(A):=\sum_{\U\in \AA_{D,p}} (-1)^{|\U|}\prod_{t=1}^{k_\U} \frac{A^{U_t}}{U_t!!}.$$
\end{definition}

\begin{remark}
Note that this set can be empty, or $\H_{D,p}^1=0$, and our result will be of no use in this case. Actually it is clear from the definition that we have $\AA_{D,p}\neq \emptyset$ if and only if there exist elements in $MI_{D,p}$ whose supports form a partition of $\Sigma$.
\end{remark}

\begin{lemma}
\label{decomp}
For each $\Theta \in \AA(\Sigma)$ as in Definition \ref{isigma}, set $\AA_{D,p}(\Theta)$ be the subset of $\AA_{D,p}$ of parts of cardinality $k_\Theta$ such that $\varphi_{U_i}=\theta_i$. Then we have the partition
$$\AA_{D,p}=\coprod_{\Theta\in \AA(\Sigma)} \AA_{D,p}(\Theta).$$
\end{lemma}

\begin{proof}
Consider the map from $\AA_{D,p}$ that sends $\U=\{U_1,\dots,U_{k_\U}\}$ to $\Theta:=\{\varphi_{U_1},\cdots,\varphi_{U_k}\}$. Since each $U_i$ is in $MI_{D,p}^\ast$, we have $\varphi_{U_i}(0)=\min \Ima \varphi_{U_i}$ from Remark \ref{minelt}. Moreover since $\U\in \AA_{D,p}$, the $\Ima \varphi_{U_i}$ form a partition of $\Sigma$. Thus the map is well defined and has image in $\AA(\Sigma)$. The decomposition in the assertion is just the decomposition of $\AA_{D,p}$ as the disjoint union of the fibers of this map.
\end{proof}

Our last task is to show the 

\begin{proposition}
\label{congl1fin}
Notations being as above, we have the congruence
$$\ell_N^{(1)}\equiv \H_{D,p}^1(A) \pi^{N(p-1)\delta} \mod \pi^{N(p-1)\delta+1}.$$
\end{proposition}

\begin{proof}
Recall from Lemma \ref{supp2} ii/, that $\ell_N^{(1)}\equiv \det A_1^\Sigma \mod \pi^{N(p-1)\delta+1}$. In other words, from Lemma \ref{decdetcyc}, we have
$$\ell_N^{(1)}\equiv \sum_{\Theta \in \AA(\Sigma)} \prod_{i=1}^{k_\Theta} M_{\theta_i}\mod \pi^{N(p-1)\delta+1}.$$

Combining this with Lemma \ref{congl1}, we get 
$$\det A_1^\Sigma\equiv \sum_{\Theta \in \AA(\Sigma)} \sum_{\U}\prod_{i=1}^{k_\Theta} (-1)^{|\U|}\frac{A^{U_i}}{U_i!!} \pi^{N(p-1)\delta} \mod \pi^{N(p-1)\delta+1},$$
where the second sum is over the parts  $\U=\{U_1,\cdots,U_{k_\Theta}\}\subset MI_{D,p}^\ast$ such that $\{\varphi_{U_1},\cdots,\varphi_{U_{k_\Theta}}\}=\{\theta_1,\dots,\theta_{k_\Theta}\}$. Thus the second sum is over $\AA_{D,p}(\Theta)$, and from lemma \ref{decomp} and the definition of $\H_{D,p}^1$, we get the announced result.
\end{proof}

All the results in this section lead us to the following

\begin{theorem}
\label{mainth}
Let $(D,p)$ be as above. Consider the hypothesis

\centerline{\rm (H)  \hspace{.5cm}   There exist elements in $MI_{D,p}$ whose supports form a partition of $\Sigma$}

Then we have the following
\begin{itemize}
	\item[i/] The first vertex of $\GNP(D,p)$ is $(N_{D,p},N_{D,p}\delta_{D,p})$ if and only if {\rm (H)} holds for $(D,p)$, and $\H_{D,p}^1$ is not the zero polynomial;
	\item[ii/] in this case the Hasse polynomial with respect to the vertex $(N_{D,p},N_{D,p}\delta_{D,p})$ is $\H_{D,p}^1$.
\end{itemize}
\end{theorem}

\begin{remark}
Since our applications are about Artin Schreier curves, we focused on the one variable case. However, the result above remains true when we have more variables.
\end{remark}

\section{Applications}

In this section, we apply Theorem \ref{mainth} to certain sets $D$ and primes $p$ satisfying (H). As a consequence, we find some known results, mainly in characteristic two, and extend them to any characteristic (Theorem \ref{T1}); we also give a new result (Theorem \ref{T2}). As a consequence of these results, we are able to show that for some genera $g$, there is no $p$-cyclic covering of the projective line with supersingular Jacobian (Corollaries \ref{noss1} and \ref{noss2}).

\subsection{Families with $d_0=p^n-1$.}

Fix a prime $p$, a positive integer $n$, and let $D:=\{1\leq i\leq p^n-1,~(i,p)=1\}$, $d_0:=p^n-1$. In this case we show

\begin{lemma}
We have 
\begin{itemize}
	\item[i/] $\delta_{D,p}=\frac{1}{n(p-1)}$;
	\item[ii/] $MI_{D,p}^\ast=\{U_0\}$, where $U_0=(u_{d})_{d \in D}$ is defined by $u_{d_0}=1$ and $u_{d}=0$ for $d\neq d_0$.
	\item[iii/] $\Sigma_{D,p}=\{1,p,\dots,p^{n-1}\}$.
\end{itemize}
\end{lemma}

\begin{proof}
First remark that from \cite[Lemma 1.3 iv/]{bl} we have $\delta_{D,p}\geq \frac{1}{s_p(D)}=\frac{1}{n(p-1)}$. Now since $d_0\equiv 0\mod p^n-1$, we have $s_{D,p}(n)=1$, and $\delta_{D,p}\leq \frac{1}{n(p-1)}$. This proves assertion i/.

We shall determine the minimal elements for $D$ and $p$. Assume $U\in E_{D,p}(m)$ is minimal. Then we have
$$\delta_p(U)=\frac{s_p(U)}{m(p-1)}=\frac{1}{n(p-1)}.$$
thus we have $m=nt$ and $s_p(U)=t$ for some positive integer $t$. Moreover, if $U=(u_{d})$, we must have $u_{d}=0$ when $d\neq d_0$ since for such $d$ we have $s_p(d)<n(p-1)$. Thus we are reduced to solve the following equation
$$u_{d_0} (p^n-1) = a (p^{nt}-1),~a>0,~ u_{d_0}\in \{ 1,\cdots,p^{nt}-1\}, ~s_p(u_{d_0})=t.$$
Now we have $u_{d_0}=a(1+p^n+\cdots+p^{(t-1)n})$, and $1\leq a \leq p^n-1$. Thus we have $s_p(u_{d_0})=ts_p(a)$, consequently $s_p(a)=1$, $a\in \{1,p,\dots,p^{n-1}\}$, and $u_{d_0}=p^i(1+p^n+\cdots+p^{(t-1)n})$. One verifies easily that $\Phi(U)=\{1,p,\dots,p^{n-1}\}$, and that $U_0$ is the unique element in $MI_{D,p}^\ast$. This proves assertion ii/ and iii/ is a direct consequence.
\end{proof}

\begin{theorem}
\label{T1}
Recall that we have set $D:=\{1\leq i\leq p^n-1,~(i,p)=1\}$, with $p$ a prime. Let $f(x)=\sum_D \alpha_d x^d\in k[x]_D$. We have
\begin{itemize}
	\item[i/] The first vertex of $\GNP(D,p)$ is $(n,\frac{1}{p-1})$.
	\item[ii/] For any $f\in k[x]_D$, the first vertex of $\NP_q(f)$ is $(n,\frac{1}{p-1})$.
\end{itemize}
\end{theorem}

\begin{proof}
From the Lemma above, $(D,p)$ satisfies (H), and we can apply Theorem \ref{mainth} to obtain i/. Since the Hasse polynomial we get is $(-1)^{n-1}a_{d_0}$, and $d_0=\max D$, we must have $a_{d_0}\neq 0$ for any polynomial in $k[x]_D$. This proves ii/.
\end{proof}

\begin{corollary}
\label{noss1}
Let $n$ be a positive integer, $p$ a prime. There is no $p$-cyclic covering of the projective line in characteristic $p$ with supersingular Jacobian, and genus
$$g=\frac{(p-1)(p^n-2)}{2},~(n,p)\notin \{(2,2),(1,3)\}.$$
\end{corollary}

\begin{proof}
Recall that the Jacobian of a curve $C$ is supersingular if and only if the Newton polygon of the numerator of the zeta function of $C$ has unique slope $\frac{1}{2}$.

A $p$-cyclic covering $C$ of the projective line in characteristic $p$ has function field defined by an equation $y^p-y=f(x)$, for $f\in k(x)$ a rational function with all its poles of order prime to $p$. By Deuring-Shafarevich formula, its $p$-rank is the number of poles of $f$ minus $1$. In order for $C$ to have supersingular Jacobian, it must have $p$-rank $0$, and $f$ has just one pole: we can assume that $f$ is a polynomial. If $C$ has genus $g_n$, $f$ must have degree $p^n-1$, and the assertion follows from Theorem \ref{T1} ii/ since the Newton polygon of the numerator of the zeta function of $C$ is the dilation by $p-1$ of $\NP_q(f)$ which has first slope $\frac{1}{n(p-1)}<\frac{1}{2}$. 
\end{proof}

\subsection{Families with $d_0=2(p^n-1)$.}

Here we fix $p$ an {\bf odd} prime, and we consider the following set 

$$D:=\{1\leq i\leq 2p^n-2,~(i,p)=1\}.$$

From above, we have $\delta_{D,p}=\frac{1}{n(p-1)}=\frac{1}{s_p(D)}$, and we look for the minimal elements. Since the subset of elements in $D$ with $p$-weight $n(p-1)$ is 
$$D_0=\{d_0,\cdots,d_n\},~d_i=2p^n-p^{i}-1,$$
we just have to look for the minimal elements in $E_{D_0,p}(m)$.

\begin{proposition}
\label{D2}
i/ Let $U$ be a minimal element in $E_{D_0,p}(m)$. Then there exists $t$ a positive integer such that $m=nt$, and we have 
\begin{itemize}
	\item $U=(p^i\frac{p^{m}-1}{p^n-1},0,\cdots,0)$, $0\leq i\leq n-1$;
	\item or $U=(0,\cdots,0,p^i\frac{p^{m}-1}{p^n-1})$, $0\leq i\leq n-1$.
\end{itemize}
ii/ We have $MI_{D,p}^\ast=\{U_1,U_2\}$, where 
\begin{itemize}
	\item $U_1=(u_{d1})$, $u_{d_n1}=1$, and $u_{d1}=0$ if $d\neq d_n$;
	\item $U_1=(u_{d2})$, $u_{d_02}=1$, and $u_{d2}=0$ if $d\neq d_0$.
\end{itemize}
As a consequence, we have $\Sigma_{D,p}=\{p^i,~0\leq i\leq n-1\} \coprod \{2p^i,~0\leq i\leq n-1\}$, $N_{D,p}=2n$, and 
$$\H_{D,p}^1=a_{d_0}a_{d_n}.$$
\end{proposition}

We begin with a lemma

\begin{lemma}
\label{sumsr}
Let $n_1,\dots,n_k$ be nonnegative integers and $u$ a positive integer such that
\begin{itemize}
	\item for each $1\leq i\leq k$, $n_i\leq p^u-1$;
	\item we have $s_p(\sum n_i)=\sum s_p(n_i)$;
\end{itemize}
then we have $\sum n_i\leq p^u-1$.
\end{lemma}

\begin{proof}
Write for each $i$: $n_i=\sum_{j=1}^{u-1} v_{ij}p^j$. Then $\sum n_i= \sum_{j=1}^{u-1} \sum_i v_{ij}p^j$; since $s_p(\sum n_i)=\sum s_p(n_i)=\sum_{j=1}^{u-1} \sum_i v_{ij}$, we must have for each $j$: $\sum_i v_{ij}\leq p-1$, and this gives the result
\end{proof}

In order to prove the first assertion of Proposition \ref{D2}, we have to solve the following congruence
$$u_0d_0+\cdots+u_nd_n\equiv 0 \mod p^m-1,~\frac{\sum_i s_p(u_i)}{m(p-1)}= \frac{1}{n(p-1)}.$$
We get a positive integer $t$ such that $\sum_i s_p(u_i)=t$, $m=nt$, and we can find some positive integer $a$ such that $\sum_i u_id_i=a(p^m-1)$. As a consequence, we have (see \cite[Proposition 11 iv/]{mm})
$$nt(p-1)\leq s_p(a(p^m-1))=s_p(u_0d_0+\cdots+u_nd_n)\leq \sum_i s_p(u_i)s_p(d_i)= nt(p-1),$$
and these inequalities are equalities. We can rewrite the sum
$$\sum_i u_id_i=\sum_{k=1}^t p^{v_k}d_{i_k},~0\leq v_k\leq nt-1.$$
Thus we have $p^{v_k}d_{i_k}\leq p^{nt-1}2(p^n-1)\leq p^{(t+1)n}-1$. From Lemma \ref{sumsr}, we have $\sum_i u_id_i\leq p^{(t+1)n}-1$, and 
$$a\leq \left\lfloor \frac{p^{(t+1)n}-1}{p^{tn}-1} \right\rfloor=p^n.$$

Now the first assertion of the proposition comes from the following lemma.

\begin{lemma}
Assume that we have written $\sum_i u_i d_i=a(p^{nt}-1)$, with $\sum_i s_p(u_i)= t$, and $1\leq a\leq p^n$; then one of the following conditions holds
\begin{itemize}
	\item $a=p^k$ for some $0\leq k\leq n-1$ and $u_n=p^k\frac{p^{nt}-1}{p^n-1}$, $u_i=0$ if $i<n$;
	\item $a=2p^k$ for some $0\leq k\leq n-1$ and $u_0=p^k\frac{p^{nt}-1}{p^n-1}$, $u_i=0$ if $i>0$.
\end{itemize}
\end{lemma}

\begin{proof}
Write 
\begin{equation}
\label{pdecomp}
a(p^{mt}-1)=\sum_{k=1}^t p^{v_k}d_{i_k},
\end{equation}
 with $v_1\leq\cdots\leq v_t$. Remark from the $p$-ary expansions of the $d_i$ that in order to have $s_p(p^{v_1}d_{i_1}+p^{v_2}d_{i_2})=2n(p-1)$, we must have
\begin{itemize}
	\item $v_2-v_1\geq n$ if $i_2=0$ or $i_1=n$;
	\item otherwise $v_2-v_1\geq n+1.$
\end{itemize}
Now reduce (\ref{pdecomp}) modulo $p^n$. we get $p^{v_1}d_{i_1}\equiv -a \mod p^n$, that is 
\begin{equation}
\label{pdecomp2}
a\equiv p^{v_1}+p^{i_1+v_1} \mod p^n
\end{equation}
Thus we have (note that the congruences become equalities since $1\leq a\leq p^n$)
\begin{itemize}
	\item[(a)] $a=p^{v_1}+p^{i_1+v_1}$ if $i_1>v_1$;
	\item[(b)] $a=p^{v_1}$ otherwise.
\end{itemize}

{\bf Case (a)} replacing $a$ by the value we got above in (\ref{pdecomp}), we obtain
$$\begin{array}{rcl}
p^{nt}a & = & p^{nt+v_1}+p^{nt+i_1+v_1}\\
& = & p^{v_t}d_{i_t}+\cdots+2p^{n+v_2}-p^{v_2+i_2}-p^{v_2}+2p^{n+v_1}.\\
\end{array}$$
Assume $t=1$; then $i_1=0$, and $v_1\leq n-1$. We get $a=2p^{v_1}$, as asserted. 

Otherwise the right hand side has valuation $nt+v_1$, thus we must have $v_2=n+v_1$, $i_2=0$. If we continue the process, we get $v_k=(k-1)n+v_1$, and $i_k=0$ for each $2\leq k\leq t$. Summing up, we obtain $p^{nt}a=2p^{nt+v_1}$. Thus $a=2p^{v_1}$, with $v_1\leq n-1$, and we must also have $i_1=0$. In other words, (\ref{pdecomp}) can be rewritten
$$2p^{v_1}(p^{mt}-1)=\left(p^{v_1}+\cdots+ p^{(t-1)n+v_1}\right)d_0$$
for some $0\leq v_1\leq n-1$.

\medskip

{\bf Case (b)} replacing $a$ by the value we obtained above in (\ref{pdecomp}), we get
$$ p^{nt+v_1}  =  p^{v_t}d_{i_t}+\cdots+2p^{n+v_2}-p^{v_2+i_2}-p^{v_2}+2p^{n+v_1}-p^{v_1+i_1}.$$
Assume $t=1$; then $i_1=n$, and $v_1\leq n-1$. We get $a=p^{v_1}$, as asserted. 

Otherwise the right hand side has valuation $nt+v_1$, thus we must have $i_1=n$ and $v_2=n+v_1$. If we continue the process, we get $i_k=n$ and $v_k=(k-1)n+v_1$ for each $2\leq k\leq t$. In other words, (\ref{pdecomp}) can be rewritten
$$p^{v_1}(p^{mt}-1)=\left(p^{v_1}+\cdots+ p^{(t-1)n+v_1}\right)d_n;$$
moreover we must have $0\leq v_1\leq n-1$ since $v_t=(t-1)n+v_1\leq tn-1$.
\end{proof}

We end the proof of Proposition \ref{D2}

\begin{proof}
It remains to show assertion ii/. The elements we have obtained from i/ have $\varphi_U(k)=p^{\overline{i+k}}$, or $\varphi_U(k)=2p^{\overline{i+k}}$, $0\leq k\leq nt-1$, where $\overline{i+k}$ is the residue of $i+k$ modulo $n$. Thus they can be irreducible iff $t=1$, and we have $\varphi_U(0)=\min \Ima \varphi_U$ iff $i=0$. This gives $MI_{D,p}^\ast$ as announced, and the other assertions are direct consequences of the definitions.
\end{proof}

\begin{theorem}
\label{T2}
Recall that we have set $D:=\{1\leq i\leq 2p^n-2,~(i,p)=1\}$, with $p$ an odd prime. Let $f(x)=\sum_D \alpha_d x^d\in k[x]_D$. We have
\begin{itemize}
	\item[i/] The first vertex of $\GNP(D,p)$ is $(2n,\frac{2}{p-1})$.
	\item[ii/] The first vertex of $\NP_q(f)$ is $(2n,\frac{2}{p-1})$ if and only if $\alpha_{d_0}\neq 0$;
	\item[iii/] Otherwise the first vertex of $\NP_q(f)$ is $(n,\frac{1}{p-1})$.
\end{itemize}
\end{theorem}

\begin{proof}
From Proposition \ref{D2}, we see that the set $D$ satisfies (H). Thus assertion i/ is a consequence of Theorem \ref{mainth}, as for assertion ii/. For iii/ we consider the subset $D'=D\backslash \{d_0\}$. Following the proof of Proposition \ref{D2}, we get that $\delta_{D',p}=\frac{1}{n(p-1)}$, $MI_{D',p}^\ast=\{U_2\}$, $N_{D',p}=n$; thus the first vertex of $\GNP(D',p)$ is as asserted, and the corresponding Hasse polynomial is $\H_{D',p}^1=(-1)^{n-1}a_{d_0}$. Since $d_0=\max D'$, we must have $\alpha_{d_0}\neq 0$ for any polynomial in $k[x]_D$. This ends the proof.
\end{proof}

\begin{remark}
Note that a consequence of the results above is the following: for $D=\{1\leq i\leq d,~(i,p)=1\}$ and $p^n-1\leq d< 2(p^n-1)$, we have $\delta_{D,p}=\frac{1}{n(p-1)}$, $MI_{D,p}^\ast=\{U_1\}$, and $\Sigma_{D,p}=\{1,\dots,p^{n-1}\}$. The first vertex is $(n,\frac{1}{p-1})$, and the Hasse polynomial is $\H_{D,p}^1=(-1)^{n-1}a_{p^n-1}$.
\end{remark}

In the same way as we deduced Corollary \ref{noss1} from Theorem \ref{T1}, we find

\begin{corollary}
\label{noss2}
Let $n$ be a positive integer, $p$ an odd prime. There is no $p$-cyclic covering of the projective line in characteristic $p$ with supersingular Jacobian, and genus
$$g=\frac{(p-1)(2p^n-3)}{2},~(n,p)\neq (1,3).$$
\end{corollary}

\subsection{The case of characteristic two.}

We only treated the case $p$ odd, since the case $p=2$ has already drawn much attention. Actually reasoning roughly the same way as in the preceding sections, we have the following precisions to the known results. Assume that $d$ is odd, and $D=\{1\leq i\leq d,~(2,i)=1\}$.

The situation is slightly different. 

When $d=2^n-1$, Theorem \ref{T1} applies without change. 

Else assume $d=2^{n+1}-3$. The $2$-density of $D$ is $\frac{1}{n}$, and the elements with $2$-weight $n$ are the $d_i=2^{n+1}-2^{i+1}-1$, $0\leq i\leq n-1$ (and $d_0=d$). Reasoning as in the proof of Proposition \ref{D2}, we get $MI(D,2)^\ast=\{U_1,U_2\}$, where $U_1$ corresponds to $d_{n-1}\equiv 0 \mod 2^n-1$, and $U_2$ to $d_{n-2}+2^{n-1}d_{0}\equiv 0 \mod 2^{2n}-1$. A rapid calculation gives $\Phi(U_1)=\{1,2,\cdots,2^{n-1}\}$ and  $\Phi(U_2)=\{1,2,\cdots,2^{n},3,3\cdot 2,\cdots,3\cdot 2^{n-2}\}$. Thus $(D,2)$ satisfies (H), and the first vertex of $\GNP(D,2)$ is $(2n,2)$ with Hasse polynomial $\H_{D,2}^1=a_{d_{n-2}}a_{d_0}^{2^{n-1}}$. Thus $\NP_q(f)$ has this vertex iff $\alpha_{d_{n-2}}\neq 0$. If $\alpha_{d_{n-2}}=0$, then it has first vertex $(n,1)$ iff $\alpha_{d_{n-1}}\neq 0$ as in Theorem \ref{T1}; when both $\alpha_{d_{n-2}}$ and $\alpha_{d_{n-1}}$ vanish, the first slope is strictly greater than $\frac{1}{n}$. 

If $2^n-1<d<2^{n+1}-3$, $\NP_q(f)$ has first vertex $(n,1)$ iff $\alpha_{d_{n-1}}\neq 0$ from above.

We summarize these results in the following Theorem, which precises \cite[Theorem 1.1 (III)]{sz1}

\begin{theorem}
Assume $f(x)=\sum_{i=1}^d \alpha_i x^i$ is a degree $d$ polynomial over $\F_q$, $q=2^m$. Then we have the following concerning the polygon $\NP_q(f)$
\begin{itemize}
	\item[i/] if $d=2^n-1$, its first vertex is $(n,1)$;
	\item[ii/] if $2^n-1<d<2^{n+1}-3$, its first vertex is $(n,1)$ iff we have $\alpha_{2^n-1}\neq 0$;
	\item[iii/] if $d=2^{n+1}-3$, its first vertex is $(2n,2)$ iff we have $\alpha_{3\cdot2^{n-1}-1}\neq 0$; if $\alpha_{3\cdot2^{n-1}-1}$ is zero, then the first vertex is $(n,1)$ iff we have $\alpha_{2^{n}-1}\neq 0$.
\end{itemize}
In the other cases, the first slope is strictly greater than $\frac{1}{n}$.
\end{theorem}

\subsection{Families with $d_0<p$.}

Here we take $D=\{1,\cdots,d_0\}$ with $d_0<p$. In this case we have $\delta_{D,p}=\frac{1}{p-1}\lceil\frac{p-1}{d_0}\rceil$. Actually, if $\sum_D du_d\equiv 0 \mod p-1$, we have $\sum_D du_d\geq p-1$, and $\sum_D u_d\geq \lceil\frac{p-1}{d_0}\rceil$. Moreover, since $\lceil\frac{p-1}{d_0}\rceil\leq p-1$, if we have $\sum_D u_d= \lceil\frac{p-1}{d_0}\rceil$, then $\sum_D s_p(u_d)= \lceil\frac{p-1}{d_0}\rceil$. Now taking $u_{d_0}=\lfloor \frac{p-1}{d_0}\rfloor$ and $u_{d_1}=1$ for $d_1= p-1-d_0\lfloor \frac{p-1}{d_0}\rfloor$, we get $\delta_{D,p}=\frac{1}{p-1}\lceil\frac{p-1}{d_0}\rceil$.

Let $U=(u_d)$ be a minimal element in $E_{D,p}(m)$. We have $\sum_D du_d\leq d_0\lceil\frac{p-1}{d_0}\rceil\leq p-2+d_0< 2p-2$. Thus we must have $m=1$, and $\Phi(U)=\{1\}$. The minimal elements for $D$ and $p$ are necessarily in $E_{D,p}(1)$, thus irreducible, and they correspond to the solutions of the system
$$\left\{\begin{array}{rcl}
\sum_D du_d & = & p-1 \\
\sum_D u_d & = & \lceil\frac{p-1}{d_0}\rceil\\
\end{array}\right.$$
One can verify easily from its definition that the Hasse polynomial is exactly $\H_{D,p}^1=\left\{f^{\lceil\frac{p-1}{d_0}\rceil}\right\}_{p-1}$, the coefficient of degree $p-1$ on $f^{\lceil\frac{p-1}{d_0}\rceil}$. This gives a new proof of \cite[Theorem 1.1]{sz3}, with the slightly better bound $d_0<p$.

\subsection{Another example in characteristic three} 
When $p=3$, consider the set $D=\{1\leq i\leq 3^{n+1}-2,~(3,i)=1\}$. The $3$-weight of $D$ is $2n+1$, we denote by $d_i=3^{n+1}-3^{i}-1$, $0\leq i\leq n$ the elements with $3$-weight $2n+1$. Reasoning as above we get that the $3$-density of $D$ is $\frac{1}{2n+1}$, and $MI(D,3)^\ast=\{U\}$, where $U$ correspond to $d_n+3^nd_0\equiv 0 \mod 3^{2n+1}-1$. A rapid calculation gives $\Phi(U)=\{1,3,\cdots,3^{n},2,2\cdot 3,\cdots,2\cdot 3^{n-1}\}$. Thus $(D,3)$ satisfies (H), and the first vertex of $\GNP(D,3)$ is $(2n+1,1)$ with Hasse polynomial $\H_{D,2}^1=a_{d_n}a_{d_{0}}^{3^{n}}$. Thus $\NP_q(f)$ has this vertex iff $\alpha_{d_n}\neq 0$.

\end{document}